\documentclass[11pt]{amsart}
\usepackage{graphicx}

\usepackage{amsmath}
\usepackage{amsthm}
\usepackage{amssymb}
\usepackage{enumerate}
\usepackage{cleveref}

\def\P{\mathcal P}

\def\C{\mathcal C}
\def\L{\mathcal L}

\newtheorem{lem}{Lemma}[section]

\newtheorem{prop}[lem]{Proposition}
\newtheorem{coro}[lem]{Corollary}
\newtheorem{thm}[lem]{Theorem}

\newtheorem{de}[lem]{Definition}

\newtheorem{pozn}[lem]{Remark}

\newtheorem{obs}[lem]{Observation}
\def\pf{\begin{proof}}
\def\pfk{\end{proof}}

\begin{document}
\title{Factor frequencies in languages invariant under more symmetries}
\maketitle
\begin{center}
\author{L{\!'}. Balkov\'a\footnote{e-mail: lubomira.balkova@fjfi.cvut.cz}\\[2mm]
{\normalsize Department of Mathematics FNSPE, Czech Technical
University in Prague}\\
{\normalsize Trojanova 13, 120 00 Praha 2, Czech Republic}}
\end{center}

\begin{abstract}
The number of frequencies of factors of length $n+1$ in a~recurrent aperiodic infinite word does not exceed
$3\Delta \C(n)$, where $\Delta \C (n)$ is the first difference of factor complexity, as shown by Boshernitzan.
Pelantov\'a together with the author derived a~better upper bound for infinite words whose language is closed under reversal.
In this paper, we further diminish the upper bound for uniformly recurrent infinite words whose language
is invariant under all elements of a~finite group of symmetries and we prove the optimality of the obtained upper bound.

\end{abstract}
%%%%%%%%%%%%%%%%%%%%%%%%%%%%%%%%%%%%%%%%%%%%%%%%%%%%%%%%%%%%%%%%%%%%%%%%%%%%%%%%%%%%%%%%%%%%%%%%%%%%%%
%%%%%%%%%%%%%%%%%%%%%%%%%%%%%%%%%%%%%%%%%%%%%%%%%%%%%%%%%%%%%%%%%%%%%%%
%%%%%%%%%%%%%%%%%%%%%%%%%%%%%%%%%%%%%%%%%%%%%%%%%%%%%%%%%%%%%%%%%%%%%%%
\section{Introduction}
When studying factor frequencies, Rauzy graph is a~powerful tool.
Using this tool, the following results have been obtained. Dekking in~\cite{De} has described factor frequencies of two famous infinite words -- the Fibonacci word and the Thue-Morse word. Using Rauzy graphs, it is readily seen that frequencies of factors of a given length of any Arnoux-Rauzy word over an $m$-letter alphabet attain at most $m+1$ distinct values.
Explicit values of factor frequencies have been derived by Berth\'e in~\cite{Be} for Sturmian words
and by Wozny and Zamboni in~\cite{WoZa} for Arnoux-Rauzy words in general.

Queff\'elec in~\cite{Qu} has explored factor frequencies of fixed points of morphisms from another point of view -- as a~shift invariant probability measure. She has provided a~rather complicated algorithm for the computation of values of such a~measure. For some special classes of fixed points of morphisms (circular marked uniform morphisms), Frid~\cite{Fr} has described completely their factor frequencies.

A~simple idea concerning Rauzy graphs lead Boshernitzan~\cite{Bo} to an upper bound on the number of different factor frequencies in
an arbitrary recurrent aperiodic infinite word. He has shown that the
number of frequencies of factors of length $n+1$ does not exceed
$3\Delta \C(n)$, where $\Delta \C (n)$ is the first difference of factor complexity.
In~\cite{Ca0}, it has been shown that $\Delta \C(n)$ is bounded for infinite words with sublinear complexity (for instance, fixed points of primitive substitutions is a~subclass of infinite words with sublinear complexity), therefore the number of different
frequencies of factors of the same length is bounded.

In our previous paper~\cite{BaPe}, making use of reflection symmetry of Rauzy graphs, we have diminished Boshernitzan's upper bound for infinite words whose language is closed under reversal.

This time, we generalize our result to infinite words whose language is invariant under all elements of a~group of symmetries and whose Rauzy graphs are therefore invariant under all elements of a~group of automorphisms.
In Section 2, we introduce basic notions, describe the main tool of our proofs -- reduced Rauzy graphs -- and summarize in detail the known upper bounds on the number of factor frequencies.
Section 3 explains what is to be understood under a~symmetry.
In Section 4, we prove Theorem~\ref{UpperBoundMoreSymmetries}, which provides an optimal upper bound on the number of factor frequencies of infinite words whose language
is invariant under all elements a~finite group of symmetries.
Section 5 is devoted to the demonstration that the upper bound from the main theorem is indeed optimal.

Finally, let us mention that the idea to exploit symmetries of the
Rauzy graph was already used in~\cite{BalMaPe} in order to estimate
the number of palindromes of a~given length, and, recently, it has been used profoundly in~\cite{PeSt1, PeSt, St}
for the generalization of the so-called rich and almost rich words (see~\cite{GlJuWiZa}) for languages invariant under more symmetries
than just reversal.

%%%%%%%%%%%%%%%%%%%%%%%%%%%%%%%%%%%%%%%%%%%%%%%%%%%%%%%%%%%%%%%%%%%%%%%%%%%%%%%%%%%%%%%%%%%%%%%%%%%%%%
%%%%%%%%%%%%%%%%%%%%%%%%%%%%%%%%%%%%%%%%%%%%%%%%%%%%%%%%%%%%%%%%%%%%%%%

\section{Preliminaries}
An {\em alphabet} $\mathcal A$ is a~finite set of symbols, called
{\em letters}. A~concatenation of letters is a~{\em word}. The {\em
length} of a~word $w$ is the number of letters in $w$ and
is denoted $|w|$. The set $\mathcal A^{*}$ of all finite words
(including the empty word $\varepsilon$) provided with the operation
of concatenation is a~free monoid. The set of all finite words but the empty word $\varepsilon$ is denoted $\mathcal A^{+}$.  We will also deal with
right-sided infinite words ${\mathbf u}=u_0u_1u_2...$, where $u_i \in {\mathcal A}$. A~finite word $w$ is
called a~{\em factor} of the word $u$ (finite or infinite) if there
exist a~finite word $p$ and a~word $s$ (finite or
infinite) such that $u=pws$. The factor $p$ is
a~{\em prefix} of $u$ and $s$ is a~{\em suffix} of $u$. An infinite
word $\mathbf u$ is said to be {\em recurrent} if each of its factors occurs
infinitely many times in $\mathbf u$. An {\em occurrence} of a~finite word $w$ in a~finite word
$v=v_1v_2\ldots v_m$ (in an infinite word $\mathbf u$) is an index $i$ such that $w$ is
a~prefix of the word $v_iv_{i+1} \ldots v_m$ (of the word $u_iu_{i+1}\dots$).
An infinite word $\mathbf u$ is called {\em uniformly recurrent} if for any factor $w$ the set
$\{j-i\ \mid  i \ \text{and} \ j \ \text{are successive occurrences of $w$ in $\mathbf u$} \}$ is bounded.
%%%%%%%%%%%%%%%%%%%%%%%%%%%%%%%%%%%%%%%%%%%%%%%%%%%%%%%%%%%%%%%%%%%%%%%%%%%%%%%%%%%%%%%%%%%%%%%%%%%%%%
%%%%%%%%%%%%%%%%%%%%%%%%%%%%%%%%%%%%%%%%%%%%%%%%%%%%%%%%%%%%%%%%%%%%%%%
\subsection{Complexity and special factors}
The {\em language} $\L(\mathbf u)$ of an infinite word $\mathbf u$ is the set of
all factors of $\mathbf u$. We denote by $\L_n(\mathbf u)$ the set of
factors of length $n$ of $\mathbf u$. We define the
{\em factor complexity} (or {\em complexity}) of $\mathbf u$ as the mapping $\C: \mathbb N
\rightarrow \mathbb N$ which associates to every $n$ the number of
different factors of length $n$ of $\mathbf u$, i.e.,
$\C(n)=\# \L_n(\mathbf u).$

An important role for the computation of factor complexity is played by
special factors. We say that a~letter $a$ is a~{\em right extension}
of a~factor $w \in \L(\mathbf u)$ if $wa$ is also a~factor of $\mathbf u$. We
denote by $\rm{Rext}(w)$ the set of all right extensions of $w$ in $\mathbf u$,
i.e., $\rm{Rext}(w)=\{a \in {\mathcal A}\mid  wa \in \L(\mathbf u)\}$.
If $\#\rm{Rext}(w)\geq 2$, then the factor $w$ is called {\em right
special} (RS for short). Analogously, we define {\em left
extensions, $\rm{Lext}(w)$, left special factors} (LS for short).
Moreover, we say that a~factor $w$ is {\em bispecial} (BS for short)
if $w$ is LS and RS.

With these notions in hand, we may introduce a~formula
for the {\em first difference of complexity} $\Delta
\C(n)=\C(n+1)-\C(n)$ (taken from~\cite{Ca1}).
\begin{equation} \label{complexity}
\Delta \C(n) \ =\ \sum_{w \in  \L_n(\mathbf u)}\bigl (\# \rm{Rext}(w)-1 \bigr
) \ =\ \sum_{w \in  \L_n(\mathbf u)}\bigl (\# \rm{Lext}(w)-1 \bigr), \quad n
\in \mathbb N.
\end{equation}
%%%%%%%%%%%%%%%%%%%%%%%%%%%%%%%%%%%%%%%%%%%%%%%%%%%%%%%%%%%%%%%%%%%%%%%%%%%%%%%%%%%%%%%%%%%%%%%%%%%%%%
%%%%%%%%%%%%%%%%%%%%%%%%%%%%%%%%%%%%%%%%%%%%%%%%%%%%%%%%%%%%%%%%%%%%%%%
\subsection{Morphisms and antimorphisms}
A~mapping $\varphi$ on ${\mathcal A}^*$ is called
\begin{itemize}
\item a~{\em morphism} if $\varphi(vw)=\varphi(v)\varphi(w)$ for any $v,w \in {\mathcal A}^*$,
\item an {\em antimorphism} if $\varphi(vw)=\varphi(w)\varphi(v)$ for any $v,w \in {\mathcal A}^*$.
\end{itemize}
We denote the set of all morphisms and antimorphisms on ${\mathcal A}^*$ by $AM(\mathcal A^*)$. Together with
composition, it forms a~monoid (the unit element is the identity mapping $\rm{Id}$).
The {\em mirror} (also called {\em reversal}) {\em mapping} $R$ defined by
$R(w_1w_2\dots w_{m-1}w_m)=w_mw_{m-1}\dots w_2w_1$
is an involutive antimorphism, i.e., $R^2=\rm{Id}$.
It is obvious that any antimorphism is a~composition of $R$ and a~morphism.

A~language $\L(\mathbf u)$ is {\em closed (invariant) under reversal} if for every
factor $w\in \L(\mathbf u)$, also its {\em mirror image}
$R(w)$ belongs to $\L(\mathbf u)$.
A factor $w$ which coincides with its mirror image $R(w)$ is called a~{\em
palindrome}.
More generally,
a~language $\L(\mathbf u)$ is {\em closed (invariant) under an antimorphism or morphism $\Psi \in AM(\mathcal A^*)$} if for every
factor $w\in \L(\mathbf u)$, also $\Psi(w)$ belongs to $\L(\mathbf u)$.
If $\theta$ is an antimorphism on ${\mathcal A}^*$, then $w=\theta(w)$ is called
a~\mbox{{\em $\theta$-palindrome}}.
It is not difficult to see that an infinite word whose language is closed under an antimorphism of finite order is recurrent.

We define the $\theta$-{\em palindromic complexity} of the infinite word $\mathbf u$ as the mapping $\P_{\theta}: \mathbb N
\rightarrow \mathbb N$ satisfying
$\P_{\theta}(n)=\# \{w \in {\mathcal L}_n(\mathbf u)\mid  w=\theta(w)\}$.
If $\theta=R$, we write $\P(n)$ instead of $\P_{R}(n)$.
Clearly, $\P(n)\leq \C(n)$ for all $n \in \mathbb N$.
A non-trivial inequality between $\P(n)$ and $\C(n)$ can be
found in~\cite{AlBaCaDa}. Here, we use a~result from
\cite{BalMaPe}.
\begin{thm}\label{odhadReversal}
If the language of an infinite word is
closed under reversal, then for all $n \in \mathbb N$, we have
 \begin{equation}\label{odhadPal1}
\P(n)+\P(n+1)  \ \leq \ \Delta \C(n) +2.
 \end{equation}
 \end{thm}
This result has been recently generalized in~\cite{PeSt}.
\begin{thm}\label{odhadPal}
Let $G \subset AM({\mathcal A}^*)$ be a~finite group containing an antimorphism and let $\mathbf u$ be an infinite word whose language is invariant under
all elements of $G$. If there exists an integer $N \in \mathbb N$ such that any factor of $\mathbf u$ of length $N$ contains all letters of $\mathcal A$, then
$$\sum_{\theta \in G^{(2)}} \bigl(\P_{\theta}(n)+\P_{\theta}(n+1)\bigr)\ \leq \ \Delta\C(n)+\#G \quad \text{for all $n \geq N,$}$$
where $G^{(2)}$ is the set of involutive antimorphism in $G$.
\end{thm}
\begin{pozn}\label{weaker_assumption} Using Remark 23 from~\cite{PeSt}, the assumption on $N$ in Theorem~\ref{odhadPal} can be replaced
with the following weaker assumption:
there exists an integer $N$ such that
\begin{enumerate}
\item for any two antimorphisms $\theta_1, \theta_2 \in G$,
it holds
$$\theta_1 \not =\theta_2 \Rightarrow \theta_1(v)\not =\theta_2(v)\quad \text{for any $v$ with $|v|\geq N$},$$

\item and for any two morphisms $\varphi_1, \varphi_2 \in G$,
it holds
$$\varphi_1 \not =\varphi_2 \Rightarrow \varphi_1(v)\not =\varphi_2(v)\quad \text{for any $v$ with $|v|\geq N$}.$$
\end{enumerate}
If $\mathbf u$ is an infinite word whose language is closed under reversal, i.e., invariant under a~morphism and an antimorphism of $G=\{\rm{Id}, R\}$,
then the above weaker assumption is satisfied already for $N=0$. Therefore, Theorem~\ref{odhadReversal} is indeed a~particular case of Theorem~\ref{odhadPal}.
\end{pozn}
%%%%%%%%%%%%%%%%%%%%%%%%%%%%%%%%%%%%%%%%%%%%%%%%%%%%%%%%%%%%%%%%%%%%%%%%%%%%%%%%%%%%%%%%%%%%%%%%%%%%%%
%%%%%%%%%%%%%%%%%%%%%%%%%%%%%%%%%%%%%%%%%%%%%%%%%%%%%%%%%%%%%%%%%%%%%%%
\subsection{Factor frequency}
If $w$ is a~factor of an infinite word $\mathbf u$ and if the following limit exists
$$\lim_{|v| \to \infty, v \in \L(\mathbf u)}
\frac {\# \{ \mbox{occurrences of $w$ in $v$} \} }{|v|}\,,$$
then it is denoted by $\rho(w)$ and called the {\em frequency} of $w$.

Let us recall a~result of Frid~\cite{Fr}, which is useful for the calculation of factor frequencies in
fixed points of primitive morphisms. In order to introduce the result, we need some
further notions. Let $\varphi$ be a~morphism on ${\mathcal A}^*=\{a_1, a_2, \dots, a_m\}^*$. We associate with $\varphi$ the {\em incidence matrix} $M_\varphi$ given by $[M_\varphi]_{ij}=|\varphi(a_j)|_{a_i}$, where $|\varphi(a_j)|_{a_i}$ denotes the number of occurrences of $a_i$ in $\varphi(a_j)$.
The morphism $\varphi$ is called {\em primitive} if there exists $k \in \mathbb N$ satisfying that the power $M_{\varphi}^k$ has all entries strictly positive. As shown in~\cite{Qu}, for fixed points of primitive morphisms,
\begin{itemize}
 \item factor frequencies exist,
 \item it follows from the Perron-Frobenius theorem that the incidence matrix has one dominant eigenvalue $\lambda$, which is larger than the modulus of any other eigenvalue,
     \item the components of the unique eigenvector $(x_1, x_2,\dots, x_m)^T$ corresponding to $\lambda$ normalized so that $\sum_{i=1}^m x_i=1$ coincide with the letter frequencies, i.e., $x_i=\rho(a_i)$ for all $i \in \{1,2,\dots,m\}$.
\end{itemize}
Let $\varphi$ be a~morphism on ${\mathcal A}^*$. We denote $\psi_{ij}:{\mathcal A}^+ \to {\mathcal A}^+$, where $i,j \in \mathbb N$, the mapping that associates with $v \in {\mathcal A}^+$ the word $\psi_{ij}(v)$ obtained from
$\varphi(v)$ by erasing $i$ letters from the left and $j$ letters from the right, where $i+j < |\varphi(v)|$.
We say that a~word $v \in {\mathcal A}^+$ admits an {\em interpretation} $s=(b_0b_1\dots b_m, i,j)$ if $v=\psi_{ij}(b_0b_1\dots b_m)$, where $b_i \in {\mathcal A}$ and $i <|\varphi(b_0)|$ and $j<|\varphi(b_m)|$. The word $a(s)=b_0b_1\dots b_m$ is an {\em ancestor} of $s$. The set of all interpretations of $v$ is denoted $I(v)$. Now we can recall the promised result of Frid~\cite{Fr}.

\begin{prop}\label{Frid}
Let $\varphi$ be a~primitive morphism having a~fixed point $\mathbf u$ and let $\lambda$ be the dominant eigenvalue of the incidence matrix $M_\varphi$. Then for any factor $v \in {\mathcal L}(\mathbf u)$, it holds
$$\rho(v)=\frac{1}{\lambda}\sum_{s \in I(v)}\rho(a(s)).$$
\end{prop}
%%%%%%%%%%%%%%%%%%%%%%%%%%%%%%%%%%%%%%%%%%%%%%%%%%%%%%%%%%%%%%%%%%%%%%%%%%%%%%%%%%%%%%%%%%%%%%%%%%%%%%
%%%%%%%%%%%%%%%%%%%%%%%%%%%%%%%%%%%%%%%%%%%%%%%%%%%%%%%%%%%%%%%%%%%%%%%
\subsection{Reduced Rauzy graphs}
Assume throughout this section that factor frequencies of infinite words in question exist.
The {\em Rauzy graph} of order $n$ of an infinite word $\mathbf u$ is
a~directed graph  $\Gamma_n$ whose set of vertices is $\L_n(\mathbf u)$ and  set of edges is $\L_{n+1}(\mathbf u)$. An edge $e = w_0
w_1 \dots w_n$ starts in the vertex $w=w_0w_1\dots w_{n-1}$, ends in
the vertex  $v=w_1\dots w_{n-1}w_n$, and is labeled by its factor
frequency $\rho(e)$.

It is easy to see that edge frequencies in a~Rauzy graph $\Gamma_n$ behave similarly as
the current in a~circuit. We may formulate an analogy of Kirchhoff's current law:
the sum of frequencies of edges ending in a~vertex equals the sum of
frequencies of edges starting in this vertex.
\begin{obs}[Kirchhoff's law for frequencies]\label{KLaw}
Let $w$ be a~factor of an infinite word $\mathbf u$ whose factor frequencies exist. Then
$$\rho(w)=\sum_{a \in \rm{Lext}(w)}\rho(aw)=\sum_{a \in \rm{Rext}(w)}\rho(wa).$$
\end{obs}
Kirchhoff's law for frequencies has some useful consequences.
\begin{coro}\label{KLaw1}
Let $w$ be a~factor of an infinite word $\mathbf u$ whose frequency exists.
\begin{itemize}
\item
If $w$ has a~unique right extension $a$, then $\rho(w)=\rho(wa)$.
\item
If $w$ has a~unique left extension $a$, then $\rho(w)=\rho(aw)$.
\end{itemize}
\end{coro}
\begin{coro}\label{KLaw2}
Let $w$ be a~factor of an aperiodic recurrent infinite word $\mathbf u$ whose frequency exists.
Let $v$ be the shortest BS factor containing $w$, then $\rho(w)=\rho(v)$.
\end{coro}
The assumption of recurrence and aperiodicity in Corollary~\ref{KLaw2} is needed in order
to ensure that every factor can be extended to a~BS factor.

Corollary~\ref{KLaw1} implies that if a~Rauzy graph contains a~vertex $w$ with only one
incoming edge $aw$ and one outgoing edge $wb$, then $\rho:=\rho(aw)=\rho(w)=\rho(wb)=\rho(awb)$.
Therefore, we can replace this triplet (edge-vertex-edge) with only one edge $awb$ keeping the frequency $\rho$.
If we reduce the Rauzy graph step by step applying the above
described procedure, we obtain the so-called {\em reduced Rauzy
graph} $\tilde{\Gamma}_n$, which simplifies the investigation of
edge frequencies. In order to precise this construction, we
introduce the notion of a~simple path.
\begin{de}\label{simple_path}
Let $\Gamma_n$ be the Rauzy graph of order $n$
of an infinite word $\mathbf u$. A~factor $e$ of length larger than $n$ such that its prefix and its suffix of length $n$ are special factors and $e$ does not contain any other special factors is called a~simple path. We define the label of a~simple path $e$ as $\rho(e)$.
\end{de}
\begin{de}\label{reduced_Rauzy_graph}
The reduced Rauzy graph
$\tilde{\Gamma}_n$ of $\mathbf u$ of order $n$ is a~directed graph whose
set of vertices is formed by LS and RS factors of $\L_n(\mathbf u)$
and whose set of edges is given in the following way. Vertices $w$
and $v$ are connected with an edge $e$ if there exists in $\Gamma_n$
a~simple path starting in $w$ and ending in $v$. We assign to such
an edge $e$ the label of the corresponding simple path.
  \end{de}
For a~recurrent word $\mathbf u$, at least one edge starts and at least one
edge ends in every vertex of $\Gamma_n$. If $\mathbf u$ is moreover aperiodic, then all its Rauzy graphs
 contain at least one LS and one RS factor. It is thus not difficult to see that for recurrent aperiodic words,
the set of edge labels in $\Gamma_n$ is equal to the set of edge labels in the reduced Rauzy graph $\tilde{\Gamma}_n$.
The number of edge labels in the Rauzy graph $\tilde{\Gamma}_n$ is clearly less
or equal to the number of edges in $\tilde{\Gamma}_n$. Let us
calculate the number of edges in $\tilde{\Gamma}_n$ in order to get
an upper bound on the number of frequencies of factors in $\L_{n+1}(\mathbf u)$.

For every RS factor $w \in \L_n(\mathbf u)$, it holds that
$\#\rm{Rext}(w)$ edges begin in $w$ and for every LS factor $v \in
\L_n(\mathbf u)$ which is not RS, only one edge begins in $v$, thus
we get the following formula
\begin{equation} \label{frequencies_vertices} \# \{ e
 \mid   e \ \mbox{edge in} \ \tilde {\Gamma}_n\}=\sum_{w \ \text{RS in $\L_n(\mathbf u)$}} \#\rm{Rext}(w)+\sum_{v \
\text{LS} \ \text{not RS in $\L_n(\mathbf u)$}}1. \end{equation}
We rewrite the first term using~\eqref{complexity} and the second term using the definition of BS factors in the following way
\begin{equation}\label{edge_vertex} \# \{ e \mid   e \ \mbox{edge in} \ \tilde {\Gamma}_n\}=\Delta C(n)+
\sum_{v \ \text{RS in $\L_n(\mathbf u)$} }1+\sum_{v \ \text{LS in $\L_n(\mathbf u)$}}1 - \sum_{v \ \text{BS in $\L_n(\mathbf u)$} }1.
\end{equation}
 Since $\#\rm{Rext}(w) - 1 \geq 1$
for any RS factor $w$ and, similarly, for LS factors, we have
\begin{equation}\label{odhadRS}
\#\{w\in \L_n(\mathbf u) \mid   w\ RS\} \ \leq  \ \ \Delta \C(n) \quad
{\rm and} \quad \#\{w\in \L_n(\mathbf u) \mid    w\ LS\} \ \leq  \ \
\Delta C(n).
\end{equation}
By combining \eqref{edge_vertex} and \eqref{odhadRS},
we obtain
\begin{equation}\label{AplusB} \# \{ e \mid   e \ \mbox{edge in} \ \tilde {\Gamma}_n\} \ \leq \ 3\Delta C(n)-X,
\end{equation}
where $X$ is the number of BS factors of length $n$.
This provides us with the result initially proved by Boshernitzan in~\cite{Bo}.
\begin{thm}\label{Boshernitzan} Let $\mathbf u$ be an aperiodic recurrent infinite word
such that the frequency
$\rho(w)$ exists for every factor $w \in \L(\mathbf u)$. Then for every $n \in \mathbb N$, it holds
\[\#\{\rho(e) \mid  e \in \L_{n+1}(\mathbf u)\}\quad \leq \quad 3\Delta
C(n).\] \end{thm}
In the paper~\cite{BaPe}, we have considered infinite words with language closed under reversal and we have lowered the upper bound from Theorem~\ref{Boshernitzan} for them.
\begin{thm}\label{UpperBoundReversalClosed}
Let $\mathbf u$ be an infinite word whose language ${\mathcal L}(\mathbf u)$ is
closed under reversal and such that the frequency $\rho(w)$ exists for every factor $w \in
{\mathcal L}(\mathbf u)$. Then for every $n
\in \mathbb N$, we have
\begin{equation}\label{BetterUpBound} \# \{\rho(e)\mid  e \in {\mathcal
L}_{n+1}(\mathbf u) \}\quad \leq \quad 2\Delta {\mathcal C}(n)+1-
\tfrac{1}{2}X-\tfrac{1}{2}Y,
\end{equation}
where $X$ is the number of BS factors of length $n$ and $Y$ is the
number of palindromic BS factors of length $n$.
\end{thm}
\begin{coro}
Let $\mathbf u$ be an infinite word whose language
${\mathcal L}(\mathbf u)$ is closed under reversal and such that the frequency $\rho(w)$ exists for every
factor $w \in {\mathcal L}(\mathbf u)$. Then the
number of distinct factor frequencies obeys for all $n \in \mathbb N$,
\begin{equation}\label{BasicBound}\# \{\rho(e)\mid  e \in {\mathcal L}_{n+1}(\mathbf u) \}\quad \leq \quad
2\Delta {\mathcal C}(n)+1,\end{equation}
where the equality is reached if and only if $\mathbf u$ is purely periodic.
\end{coro}
As shown by Ferenczi and Zamboni~\cite{FeZa}, $m$-iet words
attain the upper bound from~\eqref{BetterUpBound} for all $n \in \mathbb N$. Since Sturmian words are
2-iet words, they reach the upper bound in~\eqref{BetterUpBound} for all $n \in \mathbb N$, too.
Consequently, the upper bound from~\eqref{BetterUpBound} is optimal and cannot be improved while preserving the
assumptions. However, as we will show in the sequel, if the language of an infinite word $\mathbf u$ is invariant under more symmetries,
the upper bound from~\eqref{BetterUpBound} may be lowered considerably.
%%%%%%%%%%%%%%%%%%%%%%%%%%%%%%%%%%%%%%%%%%%%%%%%%%%%%%%%%%%%%%%%%%%%%%%%%%%%%%%%%%%%%%%%%%%%%%%%%%%%%%
%%%%%%%%%%%%%%%%%%%%%%%%%%%%%%%%%%%%%%%%%%%%%%%%%%%%%%%%%%%%%%%%%%%%%%%
\section{Symmetries preserving factor frequency}
We will be interested in symmetries preserving in a~certain way factor occurrences in $\mathbf u$, and consequently, frequencies of factors of $\mathbf u$.
Let us call a~{\em symmetry} on ${\mathcal A}^*$ any mapping $\Psi$ satisfying the following two properties:
\begin{enumerate}
\item $\Psi$ is a~bijection: ${\mathcal A}^{*}\to {\mathcal A}^{*}$,
\item for all $w,v \in {\mathcal A}^{*}$
$$\#\{\text{occurrences of $w$ in $v$}\}=\#\{\text{occurrences of $\Psi(w)$ in $\Psi(v)$}\}.$$
\end{enumerate}
\begin{thm}\label{reflection_letterpermutation}
Let $\Psi: {\mathcal A}^*\to {\mathcal A}^*$. Then $\Psi$ is a~symmetry if and only if $\Psi$ is a~morphism or an antimorphism such that $\Psi$ is a~letter permutation when restricted to $\mathcal A$.
\end{thm}
The proof of Theorem~\ref{reflection_letterpermutation} is obtained when putting together the following two lemmas.
\begin{lem}\label{bijection_n}
Let $\Psi$ be a~symmetry on ${\mathcal A}^{*}$ and let $w \in {\mathcal A}^*$. Then
$|\Psi(w)|=|w|$.
\end{lem}
\begin{proof}
Since $\#\{\text{occurrences of $\Psi(w)$ in
$\Psi(\varepsilon)$}\}=\#\{\text{occurrences of $w$ in
$\varepsilon$}\}=0$ for every $w \in {\mathcal A}^{*}$, it follows that $\Psi(\varepsilon)=\varepsilon$.

Since $\Psi$ is a~bijection, for every letter $a \in \mathcal A$, there exists a~unique $w \in {\mathcal A}^{*}$ such
that $\Psi(w)=a$, where $w \not = \varepsilon$. If we denote ${\mathcal A}=\{a_1, \dots, a_m\}$, then using Property~(2), it is easy to show that there exists a~permutation $\pi \in S_m$ such that $\Psi(a_k)=a_{\pi(k)}$ for all $k \in \{1,\dots,
m\}$.

Let us now take an arbitrary $w \in {\mathcal A}^{*}$, then
using the fact that $\Psi$ restricted to $\mathcal A$ is a~letter permutation and applying Property~$(2)$, we have
$$|w|=\sum_{a \in {\mathcal A}}\#\{\text{occurrences of $a$ in $w$}\}=
\sum_{a \in {\mathcal A}}\#\{\text{occurrences of $\Psi(a)$ in $\Psi(w)$}\}=|\Psi(w)|.$$
\end{proof}

Using Lemma~\ref{bijection_n} and the definition of symmetry, it is seen
for every $w_1w_2\dots w_n \in {\mathcal A}^*$, $w_i \in {\mathcal A}$, that the following equation is valid
\begin{equation}\label{symmEq}
\Psi(w_1w_2\dots w_n)=\Psi(w_{\sigma(1)})\Psi(w_{\sigma(2)})\dots \Psi(w_{\sigma(n)})
\end{equation}
for some permutation $\sigma \in S_n$.
The next lemma claims that the permutation $\sigma$ is necessarily either
the identical permutation $(1\ 2\ \dots \ n)$ or the symmetric permutation $(n\ \dots \ 2\ 1)$.

\begin{lem}\label{all_bijections}
Let $\Psi$ be a~symmetry on ${\mathcal A}^*$. Then $\Psi$ is either a~morphism or an antimorphism.
\end{lem}
\begin{proof}
We have to prove that
$\Psi(w)=\Psi(w_1)\Psi(w_2)\dots \Psi(w_n)$ for every $w=w_1 w_2 \dots w_n \in
{\mathcal A}^*$, $w_i \in {\mathcal A}$, or $\Psi(w)=\Psi(w_n)\dots \Psi(w_2)\Psi(w_1)$ for every
$w=w_1 w_2 \dots w_n \in {\mathcal A}^*$, $w_i \in {\mathcal A}$.

Let us proceed by induction on the length $n$ of $w$. The case $n=1$ is clear.
Suppose that $\Psi(w)=\Psi(w_1)\Psi(w_2)\dots \Psi(w_{n-1})$ for every $w=w_1 w_2 \dots w_{n-1} \in {\mathcal A}^*$ of length $n-1, \ n\geq 2$. Take an arbitrary word $w=w_1
w_2\dots w_n \in {\mathcal A}^*$. Then, as $\Psi$ is a~symmetry, $\Psi(w_2\dots w_n)$ is a~factor
of $\Psi(w_1w_2\dots w_n)$, in more precise terms, $\Psi(w_2\dots w_n)$ is either a~prefix or a~suffix of $\Psi(w_1\dots w_n)$.
Moreover, if $w_1$ occurs in $w_2\dots w_n$ $\ell$ times, $w_1$ occurs in $w_1w_2\dots w_n$ $(\ell+1)$ times.
Since $\Psi$ is a~symmetry, it follows that $\Psi(w_1)$ occurs $\ell$ times in $\Psi(w_2\dots w_n)$ and $(\ell+1)$ times in $\Psi(w_1w_2\dots w_n)$.
These two observations result in
$$\Psi(w_1 w_2\dots w_n)=\Psi(w_1)\Psi( w_2\dots w_n)=\Psi(w_1)\Psi( w_2)\dots\Psi(w_n) $$
  $$\text{or}$$
  $$\Psi(w_1 w_2\dots  w_n)=\Psi(w_2\dots w_n)\Psi(w_1)=\Psi(w_2)\dots \Psi(w_n)\Psi(w_1).$$
The first case means that $\Psi$ is a~morphism. Let us treat the second case.
Similar reasoning as before leads to
$$\Psi(w_1 w_2 \dots w_n)=\Psi(w_1\dots w_{n-1})\Psi(w_n)=\Psi(w_1)\Psi( w_2)\dots\Psi(w_n)$$
 $$\text{or}$$
 $$\Psi(w_1 w_2\dots w_n)=\Psi(w_n) \Psi(w_1\dots w_{n-1})=\Psi(w_n) \Psi(w_1)\dots \Psi(w_{n-1}).$$
The first case again means that $\Psi$ is a~morphism.
The only case which remains is $\Psi(w)=\Psi(w_2)\dots\Psi(w_n)\Psi(w_1)=\Psi(w_n) \Psi(w_1)\dots\Psi(w_{n-1})$.
Since $\Psi$ is a~bijection, we get $w_1=w_2=\dots = w_n$. Hence, again $\Psi(w)=\Psi(w_1)\Psi(w_2)\dots \Psi(w_n)$.

With the same reasoning, we deduce that if $\Psi(w)=\Psi(w_{n-1})\dots
\Psi(w_2)\Psi(w_1)$ for every $w=w_1 w_2 \dots w_{n-1} \in {\mathcal
A}^*, \ n\geq 2$, then for an arbitrary $w=w_1 w_2 \dots w_{n} \in {\mathcal A}^*$, $w_i \in {\mathcal A}$, we get $\Psi(w)=\Psi(w_n)\dots \Psi(w_2)\Psi(w_1)$.

\end{proof}
\begin{obs}\label{same_frequency}
Let $\mathbf u$ be an infinite word whose language is invariant under a~symmetry $\Psi$.
For every $w$ in $\L(\mathbf u)$ whose frequency exists, it holds
$$\rho(w)=\rho(\Psi(w)).$$
\end{obs}
\begin{pozn}
If a~finite set $G$ is a~submonoid of $AM({\mathcal A}^*)$, then $G$ is a~group and any its member restricted to the set of words of length one is just a~permutation on the alphabet $\mathcal A$. In other words, $G$ is a~finite group of symmetries. Words with languages invariant under all elements of such a~group $G$ of symmetries
have been studied in~\cite{PeSt}.
\end{pozn}
%%%%%%%%%%%%%%%%%%%%%%%%%%%%%%%%%%%%%%%%%%%%%%%%%%%%%%%%%%%%%%%%%%%%%%%%%%%%%%%%%%%%%%%%%%%%%%%%%%%%%%
%%%%%%%%%%%%%%%%%%%%%%%%%%%%%%%%%%%%%%%%%%%%%%%%%%%%%%%%%%%%%%%%%%%%%%%
\section{Factor frequencies of languages invariant under more symmetries}
Assume $\mathbf u$ is an infinite word over an alphabet $\mathcal A$ with $\#{\mathcal A}\geq 2$ whose language is invariant under all elements of a~finite group $G \subset AM({\mathcal A}^*)$ of symmetries containing an antimorphism.
Let us summarize some observations concerning the group $G$ of symmetries and reduced Rauzy graphs of $\mathbf u$. These observations constitute all tools we need for the proof of the main theorem of this paper - Theorem~\ref{UpperBoundMoreSymmetries}.\\
\noindent {\bf Observations}:
\begin{enumerate}
\item Let $\theta$ be an antimorphism in $G$. The mapping $\Psi \to \theta\Psi$ is a~bijection on $G$ satisfying
$$\text{$\Psi \in G$ is a~morphism \ $\Leftrightarrow$ \ $\theta\Psi\in G$ is an antimorphism.}$$
This implies that $G$ containing an antimorphism has an even number of elements, i.e., $\#G=2k$.
\item For a~factor $w$ containing all letters of $\mathcal A$, the following properties can be easily verified:
\begin{enumerate}
\item for any distinct antimorphisms $\theta_1, \theta_2 \in G$, we have $\theta_1(w)\not =\theta_2(w)$,
\item for any distinct morphisms $\varphi_1, \varphi_2 \in G$, we have $\varphi_1(w)\not =\varphi_2(w)$.
\end{enumerate}
\item If $w$ is a~$\theta$-palindrome containing all letters of $\mathcal A$ for an antimorphism $\theta \in G$, then $\theta$ is an involution, i.e., $\theta^2=\rm{Id}$.
\item In a~reduced Rauzy graph of $\mathbf u$, if there is an edge $e$ between two vertices $w$ and $v$, where $w$ and $v$ contain all letters of $\mathcal A$, then
    \begin{enumerate}
    \item either $e$ is a~$\theta$-palindrome for some antimorphism $\theta \in G$, then there exist at least $k$ distinct edges having the same label $\rho(e)$, namely edges
$\varphi(e)$ for all morphisms in $G$;
\item or $e$ is not a~$\theta$-palindrome for any antimorphism $\theta \in G$, then there exist at least $2k$ distinct edges having the same label $\rho(e)$, namely edges
$\varphi(e)$ for all morphisms in $G$ and $\theta(e)$ for all antimorphisms in $G$.
\end{enumerate}
\item On one hand, if an edge $e$ in the reduced Rauzy graph $\tilde{\Gamma}_n$ is mapped by
$\theta$ onto itself, then the corresponding simple path has
a~$\theta$-palindromic central factor of length $n$ or $n+1$. On the other hand, every
$\theta$-palindrome contained in ${\mathcal L}_{n+1}(\mathbf u)$ is the central factor of a~simple path
mapped by $\theta$ onto itself and every $\theta$-palindrome of length $n$ is
either the central factor of a~simple path mapped by $\theta$ onto
itself or is a~special factor (thus, evidently, a~BS factor).
\end{enumerate}

\begin{thm}\label{UpperBoundMoreSymmetries} Let $G \subset AM({\mathcal A}^*)$ be a~finite group containing an antimorphism and let $\mathbf u$ be a~uniformly recurrent aperiodic infinite word whose language is invariant under all elements of $G$ and such that the frequency $\rho(w)$ exists for every factor $w \in \L(\mathbf u)$.
Then there exists $N \in \mathbb N$ such that
\[\# \{\rho(e)\mid  e
\in \L_{n+1}(\mathbf u) \}\quad \leq \quad \frac{1}{\#G}\Bigl(4\Delta \C(n)+\#G- X-Y \Bigr) \quad \quad \text{for all $n \geq N$},\]
where $X$ is the number of BS factors of length $n$ and $Y$ is
the number of BS factors of length $n$ that are $\theta$-palindromes for an antimorphism $\theta \in G$.
\end{thm}
\begin{proof}
Since $\mathbf u$ is uniformly recurrent, we can find $N$ such that any factor of length $N$ contains all letters of $\mathbf u$. Let $\tilde{\Gamma}_n$ be the reduced Rauzy graph of $\mathbf u$ of order $n \geq N$.
We know already that the set of edge labels of $\tilde{\Gamma}_n$ is
equal to the set of edge labels of $\Gamma_n$. It is easy to see that any element of $G$ is an automorphism of $\tilde{\Gamma}_n$, i.e., $G$ maps the graph $\tilde{\Gamma}_n$ onto itself.

Let us denote by $A$ the number of edges $e$ in $\tilde{\Gamma}_n$ such that $e$ is mapped by a~certain antimorphism of $G$ onto itself (such an antimorphism is involutive by Observation~$(3)$)
and by $B$ the number of edges $e$ in $\tilde{\Gamma}_n$ such that
$e$ is not mapped by any antimorphism of $G$ onto itself, then
\begin{equation}\label{AplusB2} \# \{ e \mid
e \ \mbox{edge in} \ \tilde {\Gamma}_n\}=A+B\leq 3\Delta{\mathcal C}(n)-X,
\end{equation}
where the upper bound is taken from~\eqref{AplusB}.
We get, using Observations (3) and (5), the following formula
\begin{equation}\label{A} A=\sum_{\theta \in G^{(2)}}\bigl(\P_{\theta}(n)+\P_{\theta}(n+1)\bigr)-\sum_{\theta \in G^{(2)}}\#\{w \in \L_n(\mathbf u) \mid  w
=\theta(w) \ \text{and $w$ BS } \}, \end{equation} where we subtract the
number of BS factors of $\L_n(\mathbf u)$ that are $\theta$-palindromes for a~certain antimorphism $\theta$, in the
statement denoted by $Y$, since they are not central factors of any
simple path. If $\# G=2k$, then for every edge $e$ in $\tilde{\Gamma}_n$ that is mapped by a~certain antimorphism $\theta \in G$ onto itself, there are at least $k$ different edges with the same label $\rho(e)$ by Observation $(4 a)$.

Now, let us turn our attention to those edges of $\tilde{\Gamma}_n$ which
are not mapped by any antimorphism of $G$ onto themselves. For every such edge $e$, at
least $2k$ edges have the same label
$\rho(e)$ by Observation $(4 b)$. These considerations lead to the following estimate
\begin{equation}\label{A_AB} \# \{  \rho(e) \mid   e \in \L_{n+1}(\mathbf u)\} \leq
\frac{1}{k}A+\tfrac{1}{2k}B=\tfrac{1}{2k}A+\tfrac{1}{2k}(A+B). \end{equation}
Putting together~\eqref{A}, \eqref{AplusB2}, \eqref{A_AB}, and Theorem~\ref{odhadPal}, the statement is proven.
\end{proof}
\begin{pozn}
If an infinite word $\mathbf u$ is closed under reversal, then $G=\{\rm{Id}, R\}$ and the new upper bound from Theorem~\ref{UpperBoundMoreSymmetries} coincides with the estimate from Theorem~\ref{UpperBoundReversalClosed}.
\end{pozn}
\begin{pozn} It is easy to show that Theorem~\ref{UpperBoundMoreSymmetries}
will stay true if we replace the assumption of uniform recurrence with the weaker (however more technical) assumption from Remark~\ref{weaker_assumption}.
\end{pozn}

Finally, if we want to have a~simpler upper bound on factor frequencies,
we can use the following one, which is slightly rougher than the estimate from Theorem~\ref{UpperBoundMoreSymmetries}.
\begin{coro}
Let $G \subset AM({\mathcal A}^*)$ be a~finite group containing an antimorphism and let $\mathbf u$ be a~uniformly recurrent infinite word whose language is invariant under all elements of $G$ and such
that the frequency $\rho(w)$ exists for every factor $w \in \L(\mathbf u)$. Then there exists $N \in \mathbb N$ such that $$\# \{\rho(e)\mid  e \in \L_{n+1}(\mathbf u) \}\quad \leq \quad \frac{4}{\#G}\Delta C(n)+1 \quad \quad \text{for all $n \geq N$}.$$ The equality holds for all sufficiently
large $n$ if and only if $\mathbf u$ is purely periodic.
\end{coro}
%%%%%%%%%%%%%%%%%%%%%%%%%%%%%%%%%%%%%%%%%%%%%%%%%%%%%%%%%%%%%%%%%%%%%%%%%%%%%%%%%%%%%%%%%%%%%%%%%%%%%%
%%%%%%%%%%%%%%%%%%%%%%%%%%%%%%%%%%%%%%%%%%%%%%%%%%%%%%%%%%%%%%%%%%%%%%%
\section{Optimality of the upper bound}
In this section, we will illustrate on an example taken from~\cite{PeSt}  that the upper bound from Theorem~\eqref{UpperBoundMoreSymmetries} is attained for every $n \in \mathbb N$, $n
\geq 1$, thus it is an optimal upper bound. The infinite word $\mathbf u$ in question is the fixed point starting in $0$, which is obtained when we iterate the primitive morphism $\varphi$ given by:
\begin{equation}\label{varphi}
\varphi(0)=0130, \quad \varphi(1)=1021, \quad \varphi(2)=102, \quad \varphi(3)=013,
\end{equation}
i.e., for all $n \in \mathbb N$, the word $\varphi^n(0)$ is a~prefix of $\mathbf u$.

The corresponding incidence matrix is of the form $$M_{\varphi}=\left(\begin{matrix} 2&1&1&1\\
1&2&1&1\\
0&1&1&0\\
1&0&0&1
\end{matrix}\right),$$
its dominant eigenvalue is $\lambda=2+\sqrt{3}$ with the corresponding normalized eigenvector $$\frac{1}{2}\left(\begin{matrix} \sqrt{3}-1\\ \sqrt{3}-1\\ 2-\sqrt{3} \\ 2-\sqrt{3}
\end{matrix}\right),$$
hence we get the letter frequencies
$$\rho(0)=\rho(1)=\frac{\sqrt{3}-1}{2}, \quad \rho(2)=\rho(3)=\frac{2-\sqrt{3}}{2}.$$
We also know that the frequencies of all factors exist because of the primitivity of $\varphi$.
In~\cite{PeSt}, the following properties of $\mathbf u$ have been shown:
\begin{enumerate}
\item The language ${\mathcal L}(\mathbf u)$ is closed under the finite group of symmetries $G=\{\rm{Id}, \theta_1, \theta_2, \theta_1\theta_2\}$, where $\theta_1, \theta_2$ are involutive antimorphisms
acting on $\mathcal A$ as follows:
    $$\theta_1: 0\to 1, 1\to 0, 2\to 2, 3 \to 3 \quad \quad \text{and}\quad \quad \theta_2:0\to 0, 1\to 1, 2\to 3, 3 \to 2.$$
\item The first increment of factor complexity satisfies $\Delta {\mathcal C}(n)=2$ for all
$n \in \mathbb N,\ n \geq 1$. Moreover, every LS factor $w$ is a~prefix for some $n \in \mathbb N$,
\begin{itemize}
\item of either $\varphi^n(0)=013010210130130\dots$ and $\rm{Lext}(w)=\{1,3\}$
\item or of $\varphi^n(1)=102101301021021\dots$ and $\rm{Lext}(w)=\{0,2\}$.
\end{itemize}
\item A factor $w$ of $\mathbf u$ is LS if and only if $\theta_i(w)$ is RS for $i \in \{1,2\}$.
\end{enumerate}
In order to find the set of frequencies of factors of any length, we need to describe BS factors of~$\mathbf u$.
By Property~(3), we deduce the relation between BS factors and $\theta_i$-palindromes.
\begin{coro}\label{BSpal} Every nonempty BS factor is a~$\theta_i$-palindrome for one of the indices $i \in \{1,2\}$.
\end{coro}
\begin{prop}\label{BS}
If $v \in {\mathcal L}(\mathbf u)$ is a~BS factor of length greater than $5$, then $v=\varphi(w)p_{w_n}$,
where $w_n$ is the last letter of $w$ and $p_0=p_2=10210$ and $p_1=p_3=01301$. Moreover, $$\rho(v)=\frac{1}{\lambda}\rho(w).$$
\end{prop}
\begin{proof}
By Property $(2)$, every LS factor of length greater than $5$ starts either in $01301$ or in $10210$. Similarly, by Property $(3)$, every RS factor ends either in $01301$ or in $10210$.
It follows from the definition of $\varphi$ in~\eqref{varphi} that there exists $w \in {\mathcal L}(\mathbf u)$ such that $v=\varphi(w)01301$
or $v=\varphi(w)10210$ and that $w$ is necessarily a~BS factor. Consider $v=\varphi(w)01301$, the second case can be treated analogously.
It is then not difficult to see that $w$ ends in $w_n=1$ or $w_n=3$, hence $p_{w_n}=01301$.
In order to prove the relation between frequencies, we need to determine the set of interpretations of $v$.
It is readily seen that the set of interpretations is
\begin{itemize}
\item $\{(w01,0,3), (w02,0,2), (w30,0,2)\}$ if $w_n=1$ or $w_n=3$,
\item $\{(w10,0,3), (w13,0,2), (w21,0,2)\}$ if $w_n=0$ or $w_n=2$.
\end{itemize}
Using Proposition~\ref{Frid}, we obtain $\rho(v)=\frac{\rho(w01)+\rho(w02)+\rho(w30)}{\lambda}=\frac{\rho(w)}{\lambda}$ if $w_n=1$ or $w_n=3$, where the last equality follows from the fact that $w$ is always followed by $01, 02$, or $30$, and similarly, $\rho(v)=\frac{\rho(w10)+\rho(w13)+\rho(w21)}{\lambda}=\frac{\rho(w)}{\lambda}$ if $w_n=0$ or $w_n=2$.
\end{proof}
Proposition~\ref{BS} implies that if we want to generate all BS factors of $\mathbf u$, then it is enough to know BS factors of length less than or equal to $5$ and
to apply the mapping $w \to \varphi(w)p_{w_n}$ on them repeatedly.
Nonempty BS factors of length less than or equal to $5$ are:
\begin{enumerate}
\item $0$ and $1$, % with $Bext(0)=\{(1,2), (3,1), (1,1)\}$ and $1$ with $Bext(1)=\{(0,3), (0,0), (2,0)\}$
\item $01$ and $10$, %with $Bext(01)=\{(3,0), (1,3), (3,3)\}$ and $10$ with $Bext(10)=\{(2,0), (1,2), (2,2)\}$
\item $01301$ and $10210$. %with $Bext(01301)=\{(3,0), (1,3), (1,0)\}$ and $10210$ with $Bext(10210)=\{(0,2), (2,1), (0,1)\}$
\end{enumerate}
The aim of the rest of this section is to show that
for any length $n \in \mathbb N, \ n \geq 1$, we have

$\#\{\rho(e)\mid  e \in {\mathcal L}_{n+1}(\mathbf u)\}=
\left\{\begin{array}{rl}
2 & \text{if} \ {\mathcal L}_n(\mathbf u) \ \text{contains a~BS factor},\\
3 & \text{otherwise}.
\end{array}\right.$

Let us draw in Figure~\ref{RauzyIllustration} reduced Rauzy graphs containing short BS factors. 
\begin{figure}[!h]\label{RauzyIllustration}
\begin{center}
\resizebox{8 cm}{!}{\includegraphics{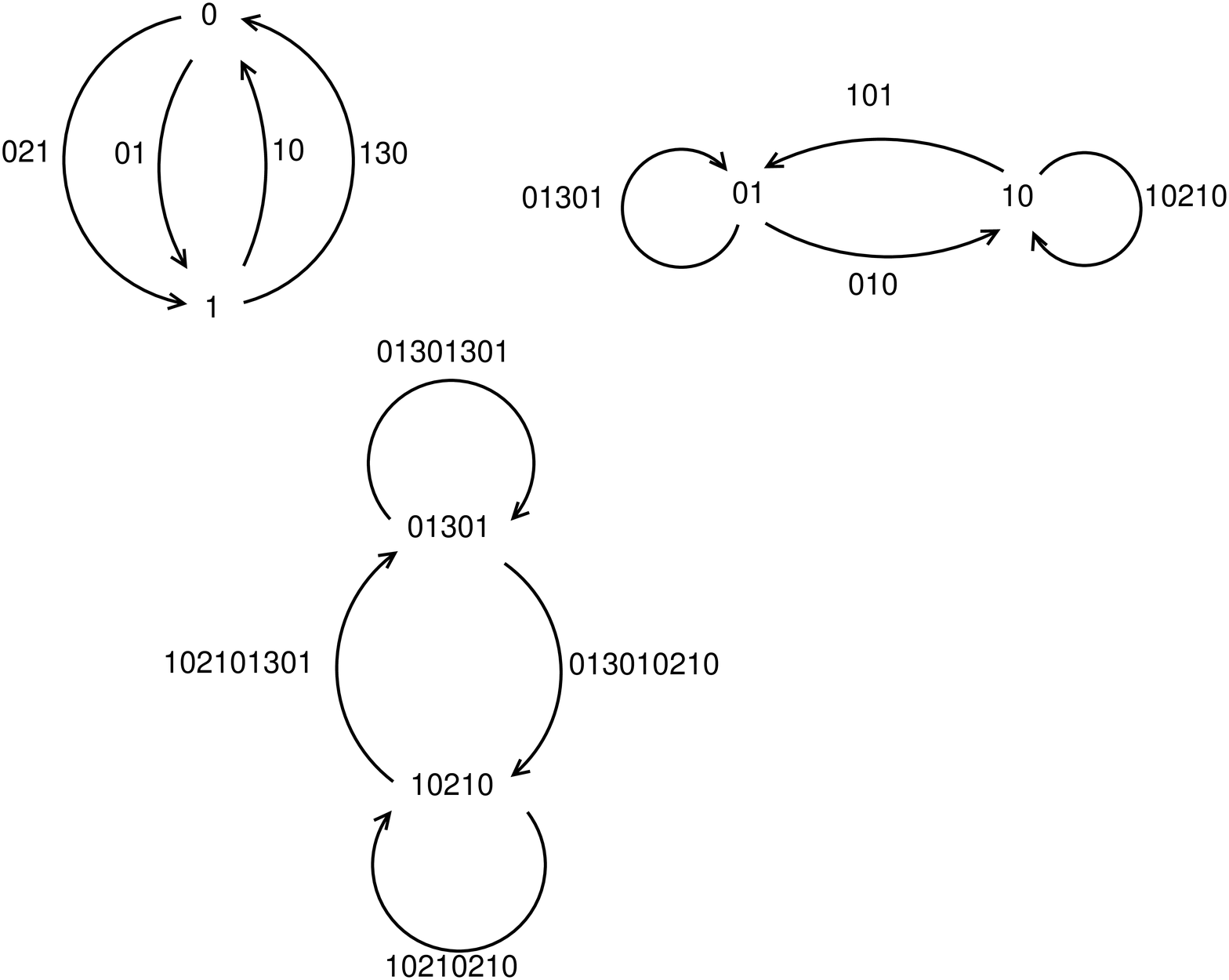}}
\end{center}
\caption{Reduced Rauzy graphs of $\mathbf u$ of order $n \in \{1,2,5\}$.}\label{RauzyIllustration}
\end{figure}
In order to describe factor frequencies, it suffices to consider reduced Rauzy graphs containing short BS factors
together with the following observations concerning reduced Rauzy graphs of $\mathbf u$.
\begin{obs}\label{redRauzy}
\begin{enumerate}
\item Any reduced Rauzy graph has either four vertices (two LS factors and two RS factors) or two vertices (BS factors).
\item Reduced Rauzy graph of larger order than 5 whose vertices are BS factors are obtained from the graphs in Figure~\ref{RauzyIllustration}
by a~repeated application of the mapping $w \to \varphi(w)p_{w_n}$ simultaneously to all vertices and edges.
\item By Corollary~\ref{KLaw2}, it is not difficult to see that if we find to a~reduced Rauzy graph $\tilde{\Gamma}_n$ whose vertices are not BS factors the reduced Rauzy graph of minimal larger order, say $\tilde{\Gamma}_m$, whose vertices are BS factors, then
    $$\{\rho(e) \mid e \ \text{edge in}\ \tilde{\Gamma}_n\}=\{\rho(e) \mid e \ \text{edge in}\ \tilde{\Gamma}_m\} \cup \{\rho(v) \mid v \ \text{vertex in} \  \tilde{\Gamma}_m\}.$$
\end{enumerate}
\end{obs}

The last step in the derivation of frequencies of factors of $\mathbf u$
is to determine the frequencies of edges and vertices in the reduced Rauzy graphs depicted in Figure~\ref{RauzyIllustration}.
In the sequel, we make use of Kirchhoff's law for frequencies (Observation~\ref{KLaw}), of the fact that symmetries preserve factor frequencies, and of the formula from Proposition~\ref{Frid}.
\begin{enumerate}
\item $\tilde{\Gamma}_1$:
$$\begin{array}{rcl}\rho(0)=\rho(1)&=&\frac{\sqrt{3}-1}{2}=\tfrac{\sqrt{3}+1}{2\lambda}\\
\rho(130)=\rho(021)&=&\rho(2)=\frac{2-\sqrt{3}}{2}=\frac{1}{2\lambda}\\
\rho(01)=\rho(10)&=&\rho(0)-\rho(130)=\frac{\sqrt{3}}{2\lambda}
\end{array}$$
In the second row, the first equality follows from the fact that symmetries preserve frequencies and $130=\theta_{2}(021)$ and the second equality by Corollary~\ref{KLaw1} from the fact that $2$ is neither LS, nor RS.
In the third row, the first equality is again due to symmetries and the second uses Kirchhoff's law for frequencies from Observation~\ref{KLaw}.
\item $\tilde{\Gamma}_2$:
$$\begin{array}{rcl}\rho(01)=\rho(10)&=&\frac{\sqrt{3}}{2\lambda}\\
\rho(01301)=\rho(10210)&=&\rho(130)=\frac{1}{2\lambda}\\
\rho(010)=\rho(101)&=&\rho(01)-\rho(01301)=\frac{\sqrt{3}-1}{2\lambda}
\end{array}$$
\item $\tilde{\Gamma}_5$:
$$\begin{array}{rcl}
\rho(01301)=\rho(10210)&=&\frac{1}{2\lambda}\\
\rho(\varphi(0)10210)=\rho(\varphi(1)01301)&=&\frac{\rho(0)}{\lambda}=\frac{\sqrt{3}-1}{2\lambda}\\
\rho(01301301)=\rho(10210210)&=&\rho(01301)-\rho(\varphi(0)10210)=\frac{2-\sqrt{3}}{2\lambda}=\frac{1}{2\lambda^2}
\end{array}$$
\end{enumerate}
Putting together Proposition~\ref{BS}, properties of reduced Rauzy graphs summarized in Observation~\ref{redRauzy}, and the knowledge of frequencies of vertices and edges in $\tilde{\Gamma}_1, \tilde{\Gamma}_2$, and $\tilde{\Gamma}_5$,
we obtain the following corollary.
\begin{coro}\label{2or3frequencies}
Let $n \in \mathbb N, \ n \geq 1,$ such that
\begin{enumerate}
\item ${\mathcal L}_{n}(\mathbf u)$ contains a~BS factor: then there exists $k \in \mathbb N$ such that the set $\{\rho(e)\mid  e \in {\mathcal L}_{n+1}(\mathbf u)\}$ is of one of the following forms:
\begin{enumerate}
\item $\{\frac{1}{2\lambda^{k+1}}, \frac{\sqrt{3}}{2\lambda^{k+1}}\}$,
\item $\{\frac{1}{2\lambda^{k+1}}, \frac{\sqrt{3}-1}{2\lambda^{k+1}}\}$,
\item $\{\frac{\sqrt{3}-1}{2\lambda^{k+1}}, \frac{1}{2\lambda^{k+2}}\}$.
\end{enumerate}
\item ${\mathcal L}_{n}(\mathbf u)$ does not contain a~BS factor: then there exists $k \in \mathbb N$ such that the set $\{\rho(e)\mid  e \in {\mathcal L}_{n+1}(\mathbf u)\}$ is of one of the following forms:
    \begin{enumerate}
\item $\{\frac{\sqrt{3}-1}{2\lambda^{k}}, \frac{1}{2\lambda^{k+1}}, \frac{\sqrt{3}}{2\lambda^{k+1}}\}$,
\item $\{\frac{\sqrt{3}}{2\lambda^{k+1}}, \frac{1}{2\lambda^{k+1}}, \frac{\sqrt{3}-1}{2\lambda^{k+1}}\}$,
\item $\{\frac{1}{2\lambda^{k+1}}, \frac{\sqrt{3}-1}{2\lambda^{k+1}}, \frac{1}{2\lambda^{k+2}}\}$.
\end{enumerate}
\end{enumerate}
\end{coro}
A~direct consequence of the previous corollary is the optimality of the upper bound from Theorem~\ref{UpperBoundMoreSymmetries}.
\begin{prop}
Let $\mathbf u$ be the fixed point of $\varphi$ defined in~\eqref{varphi}. Then for every $n \in \mathbb N, n\geq 1$, it holds
$$\# \{\rho(e)\mid  e \in \L_{n+1}(\mathbf u) \}\quad = \quad \frac{1}{\#G}\Bigl(4\Delta C(n)+\#G-X-Y\Bigr),
$$
where $X$ is the number of BS factors of length $n$ and $Y$ is
the number of BS factors of length $n$ that are $\theta_1$- or $\theta_2$-palindromes.
\end{prop}
\begin{proof}
Let us consider at first $n$ such that ${\mathcal L}_n(\mathbf u)$ does not contain a~BS factor.
Then, on one hand, Corollary~\ref{2or3frequencies} states that $\# \{\rho(e)\mid  e \in \L_{n+1}(\mathbf u)\}=3$.
On the other hand, $\frac{1}{\#G}\Bigl(4\Delta C(n)+\#G-X-Y\Bigr)=\frac{4\cdot 2+4-0-0}{4}=3$.
At second, let ${\mathcal L}_n(\mathbf u)$ contain a~BS factor.
Then, on one hand, we have by Corollary~\ref{2or3frequencies} $\# \{\rho(e)\mid  e \in \L_{n+1}(\mathbf u)\}=2$.
On the other hand, by (1) of Observation~\ref{redRauzy}, ${\mathcal L}_n(\mathbf u)$ contains 2 BS factors, and by Corollary~\ref{BSpal}, one BS factor is a~$\theta_1$-palindrome and one BS factor is a~$\theta_2$-palindrome, thus $\frac{1}{\#G}\Bigl(4\Delta C(n)+\#G-X-Y\Bigr)=\frac{4\cdot 2+4-2-2}{4}=2$.
\end{proof}
\begin{pozn}
There are also infinite words whose language is invariant under elements of a~finite group of symmetries, however, the upper bound
from Theorem~\ref{UpperBoundMoreSymmetries} is not reached for any $n \in \mathbb N$. Such an example is the famous Thue-Morse word.
Its group of symmetries $G=\{\rm{Id}, R, \Psi, \Psi\circ R\}$, where $\Psi$ is a~morphism acting on $\{0,1\}$ as follows:
$$\Psi: 0 \to 1, \ 1 \to 0.$$
As shown by Dekking~\cite{De}, the Thue-Morse word ${\mathbf u}_{TM}$ satisfies for $n \in \mathbb N, \ n\geq 1,$ $$\#\{\rho(e) \mid e \in {\mathcal L}_{n+1}({\mathbf u}_{TM})\}=\left\{\begin{array}{rl}
1 & \text{if ${\mathbf u}_{TM}$ contains a~BS factor of length $n$,}\\
2 & \text{otherwise}.
\end{array}\right.$$
But, the upper bound from Theorem~\ref{UpperBoundMoreSymmetries} is of the following form for $n \in \mathbb N, \ n\geq 1,$
$$\begin{array}{rl}
2 \ \text{or} \ 4 & \text{if ${\mathbf u}_{TM}$ contains a~BS factor of length $n$,}\\
3 \ \text{or}\ 5& \text{otherwise}.
\end{array} $$
\end{pozn}
%%%%%%%%%%%%%%%%%%%%%%%%%%%%%%%%%%%%%%%%%%%%%%%%%%%%%%%%%%%%%%%%%%%%%%%%%%%%%%%%%%%%%%%%%%%%%%%%%%%%%%
%%%%%%%%%%%%%%%%%%%%%%%%%%%%%%%%%%%%%%%%%%%%%%%%%%%%%%%%%%%%%%%%%%%%%%%
%%%%%%%%%%%%%%%%%%%%%%%%%%%%%%%%%%%%%%%%%%%%%%%%%%%%%%%%%%%%%%%%%%%%%%%
\section{Acknowledgement}
I would like to thank E. Pelantov\'a and \v S. Starosta for careful reviewing and useful remarks.
I acknowledge financial support by the Czech Science Foundation
grant 201/09/0584, by the grants MSM6840770039 and LC06002 of the
Ministry of Education, Youth, and Sports of the Czech Republic.
%%%%%%%%%%%%%%%%%%%%%%%%%%%%%%%%%%%%%%%%%%%%%%%%%%%%%%%%%%%%%%%%%%%%%%%%%%%%%%%%%%%%%%%%%%%%%%%%%%%%%%
%%%%%%%%%%%%%%%%%%%%%%%%%%%%%%%%%%%%%%%%%%%%%%%%%%%%%%%%%%%%%%%%%%%%%%%
%%%%%%%%%%%%%%%%%%%%%%%%%%%%%%%%%%%%%%%%%%%%%%%%%%%%%%%%%%%%%%%%%%%%%%%

\end{document}